\documentclass[11pt,reqno,letterpaper]{amsart}
\usepackage{amssymb}
\usepackage{graphicx,color}
\usepackage{hyperref}
\usepackage[normalem]{ulem}
\usepackage{fullpage} 

\usepackage{scalerel}[2014/03/10]
\usepackage[usestackEOL]{stackengine}

\def\dashint{\,\ThisStyle{\ensurestackMath{%
  \stackinset{c}{.2\LMpt}{c}{.5\LMpt}{\SavedStyle-}{\SavedStyle\phantom{\int}}}%
  \setbox0=\hbox{$\SavedStyle\int\,$}\kern-\wd0}\int}
\def\ddashint{\,\ThisStyle{\ensurestackMath{%
  \stackinset{c}{.2\LMpt}{c}{.5\LMpt+.2\LMex}{\SavedStyle-}{%
    \stackinset{c}{.2\LMpt}{c}{.5\LMpt-.2\LMex}{\SavedStyle-}{%
      \SavedStyle\phantom{\int}}}}\setbox0=\hbox{$\SavedStyle\int\,$}\kern-\wd0}\int}

\theoremstyle{plain}

\begin{document}


\theoremstyle{plain}
\newtheorem{theorem}{Theorem} [section]
\newtheorem{corollary}[theorem]{Corollary}
\newtheorem{lemma}[theorem]{Lemma}
\newtheorem{proposition}[theorem]{Proposition}
\newtheorem{example}[theorem]{Example}


\theoremstyle{definition}
\newtheorem{definition}[theorem]{Definition}
\theoremstyle{remark}
\newtheorem{remark}[theorem]{Remark}

\def\RR{{\mathbb R}}
\def\NN{{\mathbb N}}
\def\N{{\mathbb N}}
\def\LE{\mathcal{E}}
\def\gabo{\textcolor{red}}
\def\fran{\textcolor{magenta}}
\def\julio{\textcolor{blue}}
\def\pp{\partial \Omega \cap \partial \Omega_{\ell}}
\def\V{\mathbb{V}}
\def\U{\mathbb{U}}

 \def\T{{\mathcal {T}}}
\def\uu{\underline{u}}
\def\vv{\underline{v}}
\numberwithin{theorem}{section}
\numberwithin{equation}{section}

\title[Coupling local and nonlocal equations]{A domain decomposition scheme for couplings between local and nonlocal equations.}

\thanks{
\\
\indent
G.A. partially supported by ANPCyT under grant PICT 2018 - 3017  (Argentina) and MATHAMSUD 22MATH-04. \newline
\indent F.M.B. partially supported by ANID through FONDECYT Project 3220254, by ANPCyT under grant PICT 2018 - 3017 (Argentina), and MATHAMSUD 22MATH-04.
 \newline
\indent J.D.R. partially supported by
CONICET grant PIP GI No 11220150100036CO
(Argentina), PICT 2018 - 3183 (Argentina), UBACyT grant 20020160100155BA (Argentina) and MATHAMSUD 22MATH-04.}

\author[G. Acosta]{Gabriel Acosta}
\address[G.~Acosta]{Departamento  de Matem\'atica, FCEyN, Universidad de Buenos Aires,
 Pabell\'on I, Ciudad Universitaria (1428),
Buenos Aires, Argentina.}
\email{gacosta@dm.uba.ar}

\author[F. Bersetche]{Francisco M. Bersetche}
\address[F.M.~Bersetche]{Departamento de Matem\'atica, Universidad T\'ecnica Federico Santa Mar\'ia, Av. España
1680, Valpara\'iso, Chile.}
\email{francisco.mastrobert@usm.cl}

\author[J. D. Rossi]{Julio D. Rossi}
\address[J.D.~Rossi]{Departamento  de Matem\'atica, FCEyN, Universidad de Buenos Aires,
 Pabell\'on I, Ciudad Universitaria (1428),
Buenos Aires, Argentina.}
\email{jrossi@dm.uba.ar}

\date{}
\maketitle

\centerline{\it \small{Dedicated to Thomas Apel on the occasion of his 60th birthday.}}

\begin{abstract}
We study a natural alternating method of Schwarz type (domain decomposition) for certain class of   couplings between local and nonlocal operators. We show that our method fits into Lion's framework and prove, as a consequence, convergence in both, the continuous and the discrete settings.
\end{abstract}

\section{Introduction}
Introduced as a theoretical tool for showing existence and uniqueness of solutions for the Dirichlet problem, the celebrated Schwarz method \cite{Schwarz}, is a popular resource for dealing with general partial differential equations not only from the theoretical point of view but also as a computational tool at the discrete level. In his original work, Schwarz deals with the Laplace equation posed on a \emph{complex} or \emph{irregular}  domain that can be decomposed into  two simple overlapping subdomains. In this setting, he proposed an iterative algorithm based on solving, in an alternating way, a Laplace equation on each subdomain with an appropriate transfer of the Dirichlet boundary conditions between both problems. 
Schwarz showed, in a not completely rigorous way (see e.g. \cite{Gander}) that the proposed alternating method converges to the solution of the original problem. Over a century later, Lions introduced in \cite{Lions} an appropriate abstract framework for Schwarz's ideas. Although the main interest of Lions was computational, through the introduction of the parallel variant of the algorithm for the discretization of the original equations, his general treatment set rigorously the theory for many classical problems under general boundary conditions.

In the present paper we show that Lions's setting applies in a direct way for couplings between local and nonlocal equations such as those introduced in \cite{ABR}.
The kind of equations that we have in mind can be described \emph{informally} as follows: for a partition of a domain $\Omega\subset \mathbb{R}^N$, ${\Omega}= \left(\overline{\Omega_\ell} \cup \overline{\Omega}_{n\ell}\right)^\circ$, $ \Omega_\ell \cap \Omega_{n\ell} = \emptyset$ and a given source $f:\Omega \mapsto \mathbb{R}$ we seek for a function $u:\RR^n\to \RR$ such that 
$$
\begin{array}{l}
\displaystyle
-f(x) = \Delta u(x) +\cdots, \qquad x \in \Omega_\ell,
\\[7pt]
\displaystyle -f(x) = \int_{\Omega_{n\ell}} J(x-y) (u(y)- u(x))\, {\rm d}y + \cdots, \qquad x \in \Omega_{n\ell},
\end{array}
$$
with Dirichlet boundary conditions
(see below for a precise formulation). Notice that here we have an equation ruled by
a local operator in one part of the domain (the usual Laplacian in this case)  while in the other part the \emph{leading} term is given by a nonlocal operator. Nonlocal equations become
quite popular nowadays since they can be used to capture phenomena that 
local operators can not reproduce, see, for example, in phase transitions \cite{BCh,W} biology (population models) \cite{CF,Hutson,Strick}, elasticity \cite{DiPaola,MenDu} and neuronal networks, \cite{Z}. These models also create the need for developing new mathematical tools, we refer to \cite{ChChRo,CERW,Cortazar-Elgueta-Rossi-Wolanski} and the book \cite{ElLibro} and the survey \cite{F}. 
Previous references on coupling strategies between local and nonlocal operators appear
in connection with
atomistic-to-continuum models \cite{Peri1,Peri2,Peri3,Sel2,Sel3} (here the coupling is motivated by
microscopic descriptions of materials; different diffusions in different sets
\cite{Bere}, the existence of a prescribed region where the local and nonlocal problems
overlap \cite{delia2,delia3,delia}, couplings using different kernels, \cite{Du};
non-smooth interfaces, \cite{Gal,Kri}. Some of these coupling strategies impose continuity
of the solution. There are also numerical methods associated with these couplings, \cite{Han}.
For extra references we refer to the survey \cite{SUR}. 
The kind of coupling that we consider here is based on an energy formulation
that avoids the use of a transition region, but solutions to this kind of model
are not necessarily continuous. This kind of coupling strategy was proposed in \cite{ABR,GQR}
where the existence of minimizers of the energy was proved both for the scalar problem and for the vectorial case.

Our main goal is to adapt Lion's ideas
(iterating the solution operator for each part of the problem)
to solve this kind of problems.
From a theoretical point of view, this approach can be used to give an alternative proof for existence and uniqueness of solutions for the considered coupled models, however, notice that the original proof of existence and uniqueness is more direct (see \cite{ABR}). However, this idea of alternating the problems to be solved shield some light on the structure of the solution (in fact, one can look at the coupled local/nonlocal problem
as a system composed by a local equation and a nonlocal one both
with coupling terms, see Remark \ref{remark.pepe}). From a numerical perspective, it gives a formal proof of the following expected and informally stated result: if we have an abstract  algorithm $A1$  delivering solutions for the \emph{local} operator and an algorithm $A2$ doing the same for the  \emph{nonlocal} counterpart, they can be simultaneously used -as black boxes- to approximate solutions of coupled models.

In order to present our ideas we need to introduce some definitions and fix basic notations.

Consider two \emph{disjoint} open bounded sets $\Omega_\ell$ and $\Omega_{n\ell}$ representing respectively a local and a nonlocal region (in a sense clarified below) and define the set  $\Omega=\left(\bar \Omega_\ell \cup \bar\Omega_{n\ell}\right)^\circ.$  The unknowns in our problems are given by functions $u:\Omega\to \mathbb{R}$ with a prescribed zero value in $\Omega^c$ (i.e. an homogeneous Dirichlet case) that minimize some appropriate energy functionals. For the sake of clarity, we split functions $w$ defined in $\Omega$  in two parts that we call $u:\Omega_{\ell}\to \mathbb{R}$ and $v:\Omega_{n\ell}\to \mathbb{R}$, where
$u=w|_{\Omega_\ell}$ and $v=w|_{\Omega_{n\ell}}$. We will always assume that
the involved functions $w$ are extended to $\mathbb{R}^N \setminus \Omega$ by zero.  
Here we mainly focus on a single model  considered in \cite{ABR} (called $(i)$ in that paper and also here), since it exemplifies how our approach can be applied to more general cases. Loosely speaking,  this model deals with a  \emph{volumetric} coupling (i.e. the only interactions allowed between the local and the nonlocal regions are those between portions of positive $n-$dimensional measure) (see Figure \ref{fig:two_models}).
\begin{figure}
\begin{center}
    \includegraphics[scale=.4]{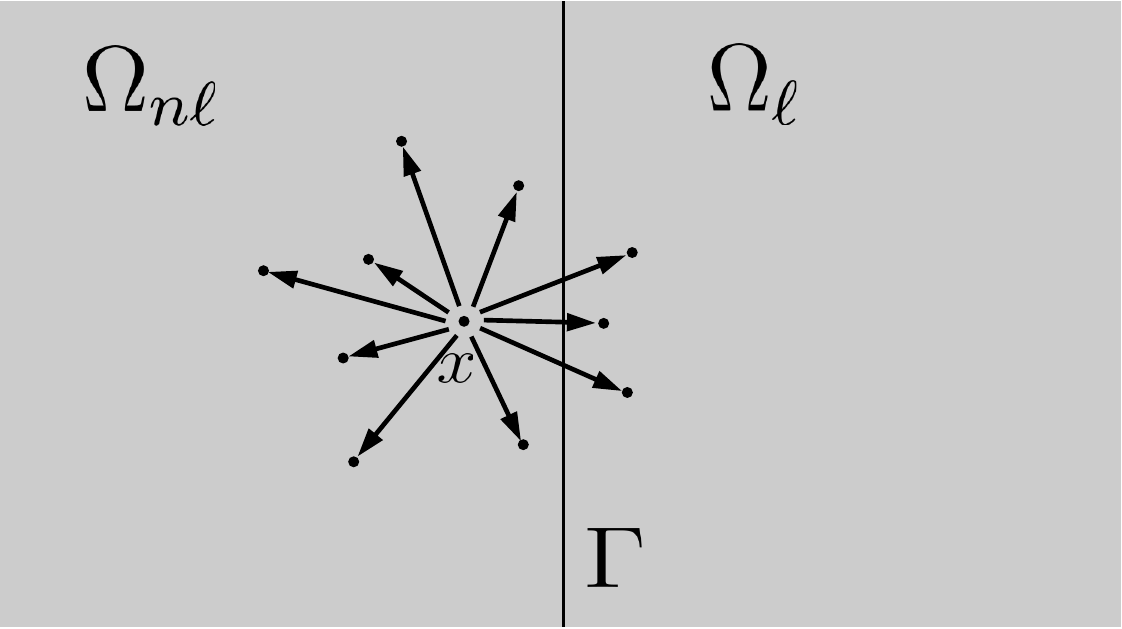}
    		\caption{{\emph{Volumetric} nonlocal interactions from $x \in \Omega_{n\ell}$. In this particular example the closures of $\Omega_{\ell}$ and $\Omega_{n\ell}$ touch each other on a common piece of boundary $\Gamma$.}}
\end{center}
\label{fig:two_models}
\end{figure}
Model $(i)$ is associated to the following energy (up to the notation, $E_i$ is the energy introduced in \cite{ABR}),
\begin{equation}\label{def:funcional_dirichlet.intro.2299}
\begin{array}{l}
\displaystyle
E_{i} (u,v):=\int_{\Omega_\ell} \frac{|\nabla u(x)|^2}{2} {\rm d}x + \frac{1}{2}\int_{\Omega_{n\ell}}\int_{\Omega_{n\ell}} J(x-y)(v(y)-v(x))^2\, {\rm d}y{\rm d}x
\\[10pt]
\qquad \qquad \displaystyle + \frac{1}{2}\int_{\Omega_{n\ell}}\int_{\Omega_{\ell}} J(x-y)(u(y)-v(x))^2\, {\rm d}y{\rm d}x
+\frac{1}{2}\int_{\Omega_{n\ell}}\int_{\Omega^c} J(x-y)v(x)^2\, {\rm d}y{\rm d}x \\[10pt] \displaystyle \qquad \qquad
- \int_\Omega f(x) u(x) {\rm d}x,
\end{array}
\end{equation}
where the  involved kernel $J:\RR^N \mapsto \RR$,  encodes the effect of a general \emph{volumetric} nonlocal  interaction inside the nonlocal part of the domain as well as between the nonlocal part with the local region (and the whole $\mathbb{R}^N$). This kernel is a nonnegative, symmetric ($J(z)=J(-z)$), bounded and integrable function with the following properties:
\begin{itemize}
  \item[(J1)]\emph{Visibility:} there exist $\delta>0$
and $C>0$ such that $J(z)>C$ for all $z$ such that $\|z\|\le 2\delta$.
\item[(J2)] \emph{Compactness:} the convolution type operator
$C(u) (x)=\int_{\Omega_{n\ell}} J(x-y)u (y) \, dy$, defines a compact operator in $L^2(\Omega_{n\ell})$.
 \end{itemize}
In our context, condition $(J1)$ guarantees the influence of nonlocality within an horizon of size at least $2\delta$ while $(J2)$ is a technical requirement fulfilled, for  instance, by continuous
kernels, characteristic functions, or even for $L^2$ kernels, (this holds since these kernels produce Hilbert-Schmidt operators of the form
$f \mapsto C(u)(x):=\int_A k(x,y) u(y) {\rm d}y$ that are compact if $k\in L^2(A\times A)$, see Chapter VI in \cite{brezis}).

Notice that the energy $E_i$ is well defined and finite for functions in the space
$$
L = \Big\{ (u,v)\in H^1 (\Omega_\ell)\times  L^2 (\Omega_{n\ell}):   u|_{\partial \Omega\cap \partial \Omega_\ell}=0\,  \Big\},
$$
where we have incorporated the Dirichlet homogeneous condition.
It is worthwhile to remark that the boundedness hypothesis that we assumed on the kernel $J$ 
says that the previously defined space $L$ coincide precisely with the space of functions with finite energy. Therefore, $L$ is the natural space to look for minimizers of the functional.

A result of existence and uniqueness of minimizers of $E_{i}$ is stated below after introducing a needed concept of connectivity.

 \begin{definition}\label{def:deltaconnected}
 {\rm We say that an open set $D \subset \mathbb{R}^N$ is $\delta-$connected , with $\delta\ge 0$, if it can not be written as a disjoint union of two
 (relatively) open nontrivial sets
 that are at distance greater or equal than $\delta.$}
\end{definition}

 Notice that if $D$ is $\delta$ connected then it is $\delta'$ connected for any $\delta'\ge \delta$. From Definition \ref{def:deltaconnected},
we notice that $0-$connectedness agrees with the classical notion of being connected (in particular, open connected sets are $\delta-$connected). Definition
\ref{def:deltaconnected} can be written in an equivalent way: an open set $D$ is $\delta-$connected if given
  two points $x,y\in D$,  there exists a finite number of points $x_0,x_1,...,x_n \in D$
 such that $x_0=x$,
 $x_n=y$ and $dist(x_i,x_{i+1}) <\delta$.

 Loosely speaking $\delta$-connectedness combined with $(J1)$ says that the effect of nonlocality can travel beyond the horizon $2\delta$ through the whole domain. For the rest of this paper we assume
 \begin{enumerate}
 \item[(1)] $ \Omega_\ell$ is connected and has Lipschitz boundary, 
 \item[(2)] $ \Omega_{n\ell} $ is $\delta-$connected.
 \end{enumerate}

Finally, in order to avoid trivial couplings in any of the two cases, we assume that  $\Omega_{\ell}$ and $\Omega_{n\ell}$ are close enough (closer than the horizon of the involved kernel)
\begin{enumerate}
\item[$(P1)$]  $dist ( \Omega_{\ell},\Omega_{n\ell})<\delta$.
\end{enumerate}

The following result is proved in \cite{ABR}.

{\bf Theorem} \label{teo.1.22}
Assume $(J1)$, $(J2)$, $(1)$, $(2)$ and $(P1)$.
Given $f \in L^2 (\Omega)$
there exists a unique minimizer $(u,v)$ of  $E_{i}$ in $L$.
The unique minimizer is a weak solution to the equation
\begin{equation}  \label{eq:main.Dirichlet.local.2277}
\begin{cases}
\displaystyle - f(x)=\Delta u (x) + \int_{\Omega_{n\ell}} J(x-y)(v(y)-u(x))\,{\rm d} y, &x\in \Omega_\ell,
\\[7pt]
\displaystyle \partial_\eta u(x)=0,\qquad & x\in  \partial \Omega_\ell \cap \Omega,  \\[7pt]
u(x)= 0, & x \in \partial \Omega \cap \partial \Omega_\ell,
\end{cases}
\end{equation}
 in $\Omega_{\ell}$ and to the following  nonlocal equation in $\Omega_{n\ell}$, with the nonlocal Dirichlet exterior condition
\begin{equation}  \label{eq:main.Dirichlet.nonlocal.2277}
\left\{
\begin{array}{l}
\displaystyle \!\! - f(x) = \!\! \int_{\mathbb{R}^N\setminus \Omega_{n\ell}} J(x-y)(\! u(y)-v(x)\! ) {\rm d}y
+ 2 \! \int_{\Omega_{n\ell}} J(x-y)  (\! v(y)-v(x) \!){\rm d}y, \\[7pt]
\qquad\qquad\qquad\qquad\qquad\qquad\qquad\qquad\qquad\qquad \qquad\qquad\qquad
\, x \in \Omega_{n\ell}, \\[7pt]
\displaystyle \!\! u(x) = 0,\qquad  x\in \mathbb{R}^N\setminus \Omega.
\end{array} \right.
\end{equation}

Notice that the subdifferential of the energy is given by 
$$
\begin{array}{l}
\displaystyle
\partial E_{i} (u,v) (\varphi, \psi):=\int_{\Omega_\ell} \nabla u(x) \nabla \varphi (x){\rm d}x + \int_{\Omega_{n\ell}}\int_{\Omega_{n\ell}} J(x-y)(v(y)-v(x))(\psi(y)-\psi(x))\, {\rm d}y{\rm d}x
\\[10pt]
\qquad \qquad \displaystyle + \int_{\Omega_{n\ell}}\int_{\Omega_{\ell}} J(x-y)(u(y)-v(x))(\varphi(y)-\psi(x))\, {\rm d}y{\rm d}x
+\int_{\Omega_{n\ell}}\int_{\Omega^c} J(x-y)v(x)\psi (x) \, {\rm d}y{\rm d}x \\[10pt] \displaystyle \qquad \qquad
- \int_{\Omega_{\ell}} f(x) \varphi (x) {\rm d}x - \int_{\Omega_{n\ell}} f(x) \psi (x) {\rm d}x,
\end{array}
$$
that is the weak form of the equations described in the previous result. Remark that
an homogeneous Neumann boundary condition for $u$ appears on the part of the boundary of $\Omega_\ell$ that is inside $\Omega$. In fact, the term
$\int_{\Omega_\ell} \nabla u(x) \nabla \varphi (x){\rm d}x$ is the weak form of $\Delta u (x)$
with mixed conditions, $u=0$ on $\partial \Omega_\ell \cap \partial \Omega$ and  
$\partial_\eta u =0$ on $\partial \Omega_\ell \cap \Omega$.

Now, with a fixed source $f$, we consider the operator that, given $v$, solves the local equation \eqref{eq:main.Dirichlet.local.2277}
and we call it $
L_{\ell}^f (v) = u$. We also introduce the operator that, given $u$, solves 
the nonlocal part of the problem \eqref{eq:main.Dirichlet.nonlocal.2277}
that we denote by $
L_{n\ell}^f (u) = v$. With these solutions
we define, for a given
$u_0$,  the alternating method $(M)$ given by the sequence
$$(M) \, \begin{cases}
v_{n+1} &= L_{n \ell}^f (u_{n}) \\
u_{n+1} &= L_{\ell}^f (v_{n+1}).
\end{cases}
$$
That is, we solve, iteratively, the local part of the equation
(fixing the nonlocal component) and the nonlocal equation
(fixing the local component).

Our main results can be summarized as follows:

\begin{theorem} \label{teo.1}
The sequence given by $(M)$ converges geometrically to the solution to
\eqref{eq:main.Dirichlet.local.2277}--\eqref{eq:main.Dirichlet.nonlocal.2277}.

Moreover, if $u_0$ is the local part of a subsolution to
\eqref{eq:main.Dirichlet.local.2277}--\eqref{eq:main.Dirichlet.nonlocal.2277}
then the sequence is monotone increasing, in the sense that $v_{n+1} \geq v_n$ and
$u_{n+1} \geq u_n$.

In addition, appropriate finite element discretizations of 
$L_{\ell}^f$ and $L_{n\ell}^f$ provide converging sequences  (obtained by alternating or parallel iterations)  to the continuous solution to
\eqref{eq:main.Dirichlet.local.2277}--\eqref{eq:main.Dirichlet.nonlocal.2277}.
\end{theorem}

The paper is organized as follows: In Section \ref{sect-alternating} we
study the alternating method in its continuous version and in Section \ref{sect-numerical}
we deal with the numerical approximations and some numerical experiments.

\section{Alternating method for $(i)$} \label{sect-alternating}

For a \emph{given} $v \in L^2 (\Omega_{n\ell})$ we define the operator
$$
L_{\ell}^f (v) = u
$$
with $u$ the solution to \eqref{eq:main.Dirichlet.local.2277}
in the local domain $\Omega_\ell$ (observe that we use $L_{\ell}^f $
we are \emph{not} claiming that $(u
,v)$ is a solution for the coupled system
\eqref{eq:main.Dirichlet.local.2277}, \eqref{eq:main.Dirichlet.nonlocal.2277}, but only that
$u$ the solution to \eqref{eq:main.Dirichlet.local.2277} for that particular $v$). In the case $f=0$ we omit the superscript and just write $L_\ell(u)$. Here we need to introduce the space
$$
H_{\partial \Omega \cap \partial \Omega_\ell}^1(\Omega_{\ell}) = \left\{
u \in H^{1} (\Omega_{\ell}) : u|_{\partial \Omega \cap \partial \Omega_\ell}=0 \right\},
$$
that is, functions in $H_{\partial \Omega \cap \partial \Omega_\ell}^1(\Omega_{\ell}) $
are functions in  $H^{1} (\Omega_{\ell})$ that vanish (in the sense of traces) on $
\partial \Omega \cap \partial \Omega_\ell$.
Notice that $L_{\ell}^f (v)=u \in H_{\partial \Omega \cap \partial \Omega_\ell}^1(\Omega_{\ell})$, can also be obtained as a minimizer of 
$$
\begin{array}{l}
\displaystyle 
E_\ell (u) = \int_{\Omega_\ell}\frac{ |\nabla u|^2 (x)}{2} {\rm d} x
+ \int_{\Omega_{\ell}} u(x)^2 \int_{\Omega_{n\ell}} J(x-y) {\rm d} y {\rm d} x
\\[10pt]
\qquad \qquad \displaystyle - \int_{\Omega_\ell} f (x) u (x) {\rm d} x - \int_{\Omega_{\ell}} u(x)  \int_{\Omega_{n\ell}} J(x-y) v(y) {\rm d} y  {\rm d} x.
\end{array}
$$

Our first lemma says that this operator $L_{\ell}^f $ is well defined and bounded. 

\begin{lemma} \label{lema.est.u}
The operator $L_{\ell}^f:L^2(\Omega_{n\ell})\to H_{\partial \Omega \cap \partial \Omega_\ell}^1(\Omega_{\ell})$ is well defined and we have
$$
\|u\|_{H^1(\Omega_\ell)}\le C \left( \|f\|_{L^2(\Omega_\ell)}+ \|v\|_{L^2(\Omega_{n\ell})}\right).
$$
In particular $L_{\ell}$ is a bounded linear operator
$$
\|L_{\ell}(v)\|_{H^1(\Omega_\ell)}\le C \|v\|_{L^2(\Omega_{n\ell})}.
$$
\end{lemma}

\begin{proof}
The energy functional $E_\ell(u)$ has a unique minimizer in $H^1_{\partial \Omega_{\ell} \cap \partial \Omega}$ (that can be obtained by the direct method of calculus of variations) that corresponds to the unique weak solution of
$$
- f(x)+ \int_{\Omega_{n\ell}} J(x-y)v(y)\,{\rm d} y=\Delta u (x) - u(x) \int_{\Omega_{n\ell}} J(x-y)\,{\rm d} y.
$$

Now, multiplying both sides of this identity by $u$, integrating in $\Omega_{\ell}$, using that the continuous function $\int_{\Omega_{n\ell}} J(x-y)\,{\rm d} y$ is nonnegative and that $\int_{\Omega_{n\ell}} J(x-y)v(y)\,{\rm d} y$ defines a continuous operator in $L^2$ we get
$$
\|u\|_{H^1(\Omega_\ell)}\le C \left( \|f\|_{L^2(\Omega_\ell)}+ \|v\|_{L^2(\Omega_{n\ell})}\right). %
$$
\end{proof}

Now, \emph{given} $u \in H^1 (\Omega_\ell)$, we let 
$$
L_{n\ell}^f (u) = v
$$
where $v\in L^2 (\Omega_\ell)$ is the solution to \eqref{eq:main.Dirichlet.nonlocal.2277} (again we are \emph{not} claiming that $(u
,v)$ is a solution for the coupled system
\eqref{eq:main.Dirichlet.local.2277}, \eqref{eq:main.Dirichlet.nonlocal.2277}). As before, we reserve the notation $L_{n\ell}$ for the case $f=0$.
Notice that $v$ can be obtained minimizing the functional
$$
\begin{array}{l}
\displaystyle
E_{n\ell} (v) = \frac12 \int_{\Omega_{n\ell}} \int_{\Omega_{n\ell}}\!\! J(x-y)(v(y)-v(x))^2\, {\rm d} y  \, {\rm d} x
+ \int_{\Omega_{n\ell}} \int_{\Omega_\ell} J(x-y) |v(x)|^2  {\rm d} x \\[10pt]
\qquad \qquad \displaystyle
- \int_{\Omega_{n\ell}} f(x) v(x) {\rm d} x -
 \int_{\Omega_{n\ell}} v(x) \int_{\Omega_{\ell}} J(x-y) u(y) {\rm d} y  {\rm d} x.
 \end{array}
 $$
 
 In this context we have Poincare type inequalities.
 
\begin{lemma} \label{lemma:poincare}
Assume $(J1)$, $(J2)$, $(2)$ and $(P1)$. Then, for any $v\in L^2(\Omega_{n\ell})$ 
such that $\bar v=\int_{\Omega_{n\ell}} v=0$ we have
$$
\|v\|_{L^2(\Omega_{n\ell})}^2\le C \int_{\Omega_{n\ell}}\int_{\Omega_{n\ell}} J(x-y)(\! v(y)-v(x)\! )^2 {\rm d}y {\rm d}x
$$
for a constant $C=C(\Omega_{n\ell})$.

In addition, for any continuous function $h(x)$,  $h(x)>0$ in a set $B$ of positive n-dimensional measure $B\subset \Omega_{n\ell}$ and any $v\in L^2(\Omega_{n\ell})$ it holds that
$$
\|v\|_{L^2(\Omega_{n\ell})}^2\le C \left( \int_{\Omega_{n\ell}}h(x)v^2(x)+ \int_{\Omega_{n\ell}}\int_{\Omega_{n\ell}} J(x-y)(\! v(y)-v(x)\! )^2 {\rm d}y {\rm d}x\right)
$$
for a constant $C=C(\Omega_{n\ell},B,h)$.
\end{lemma}

\begin{proof}
 Both inequalities can be obtained arguing as in \cite{ABR} (Section 2), see also \cite{ElLibro} (Chapter 3).
 For example, for the second inequality we can proceed as follows: assume that the inequality
 does not hold. Then we have a sequence $v_n$ such that
 $$
\|v_n\|_{L^2(\Omega_{n\ell})}^2 = 1
$$
and 
$$
\int_{\Omega_{n\ell}}h(x)v_n^2(x) {\rm d}x \to 0  $$
$$
\int_{\Omega_{n\ell}}\int_{\Omega_{n\ell}} J(x-y)(\! v_n(y)-v_n(x)\! )^2 {\rm d}y {\rm d}x
\to 0.
$$
From the last convergences (using results from \cite{ElLibro}) we obtain that
$v_n$ converges strongly in $L^2(\Omega_{n\ell})$ to a limit $w$ that is a constant 
in $\Omega_{n\ell}$ (here we are suing that the domain in $\delta$--connected)
and must verify
$$
w^2 \left( \int_{\Omega_{n\ell}}h(x) {\rm d}x  \right) = 0 .
$$
This implies that $w \equiv 0$ and hence 
we arrive to a contradiction since we have
 $$
0= \|w\|_{L^2(\Omega_{n\ell})}^2= \lim_n \|v_n\|_{L^2(\Omega_{n\ell})}^2 = 1.
$$
 \end{proof}

 We also have an analogous result to Lemma \ref{lema.est.u}.

\begin{lemma} \label{lema.est.v}
For the operator $L_{n\ell}^f: H_{\partial \Omega \cap \partial \Omega_\ell}^1(\Omega_{\ell})\to L^2(\Omega_{n\ell})$ we have that is well defined with
$$
\|v\|_{L^2(\Omega_{n\ell})}\le C \left( \|f\|_{L^2(\Omega_\ell)}+ \|u\|_{H^1(\Omega_\ell)}\right).
$$
In particular $L_{n\ell}$ is a bounded linear operator
$$
\|L_{n \ell}(u)\|_{L^2(\Omega_{n\ell})}\le C \|u\|_{H^1(\Omega_\ell)}.
$$
\end{lemma}

\begin{proof} As before, existence and uniqueness of $v$
the solution to \eqref{eq:main.Dirichlet.nonlocal.2277} can be obtained
minimizing the energy $E_{n\ell} $.

Multiplying  both sides of \eqref{eq:main.Dirichlet.nonlocal.2277} by $v(x)$  and integrating in $\Omega_{n\ell}$ yields
$$
\begin{array}{l}
\displaystyle \!\! -\int_{\Omega_{n\ell}} f(x)v(x) = \!\! \int_{\Omega_{n\ell}}v(x)\int_{\Omega_{n\ell}^c} J(x-y)(\! u(y)-v(x)\! ) {\rm d}y {\rm d}x
\\[10pt]
\qquad \qquad\qquad\qquad \displaystyle -  \! \int_{\Omega_{n\ell}} \int_{\Omega_{n\ell}} J(x-y)  (\! v(y)-v(x) \!)^2{\rm d}y {\rm d}y,
\end{array}$$
where in the last term we used the symmetry of the double integral given by the fact $J(z)=J(-z)$. Since
\begin{align*}
\displaystyle \int_{\Omega_{n\ell}}v(x)\int_{\Omega_{n\ell}^c} J(x-y)(\! u(y)-v(x)\! ) {\rm d}y {\rm d}x&= \int_{\Omega_{n\ell}}v(x)\int_{\Omega_{\ell}} J(x-y)\!\ u(y) {\rm d}y {\rm d}x \\
&- \int_{\Omega_{n\ell}}v(x)^2\int_{\Omega_{n\ell}^c} J(x-y) {\rm d}y\,  {\rm d}x
\end{align*}
in particular, calling $h(x)$ the continuous function $$0\le h(x)=\int_{\Omega_{n\ell}^c} J(x-y) {\rm d}y$$ 
that is finite since we assumed that $J$ is integrable, we obtain
\begin{align*}
\displaystyle \int_{\Omega_{n\ell}}v(x)^2h(x)\,  {\rm d}x &+ \! \int_{\Omega_{n\ell}} \int_{\Omega_{n\ell}} J(x-y)  (\! v(y)-v(x) \!)^2{\rm d}y {\rm d}y \\
=& \int_{\Omega_{n\ell}} f(x)v(x) + \int_{\Omega_{n\ell}}v(x)\int_{\Omega_{\ell}} J(x-y)\!\ u(y) {\rm d}y {\rm d}x \\
\le & \|v\|_{L^2(\Omega_{n\ell})}\left(\|f\|_{L^2(\Omega_{n\ell})}+ \|u\|_{L^2(\Omega_{\ell})} \right).
\end{align*}
Since $h(x)>0$ in a set $B$ of positive n-dimensional measure $B\subset \Omega_{n\ell}$ we can use Lemma \ref{lemma:poincare} and obtain
$$
\|v\|_{L^2(\Omega_{n\ell})}\le C \left(\|f\|_{L^2(\Omega_{n\ell})}+ \|u\|_{L^2(\Omega_{\ell})} \right).
$$
\end{proof}

\begin{remark}
Notice that it is possible to give two different proofs of Lemma \ref{lema.est.u} (resp. Lemma \ref{lema.est.v}); one showing existence and uniqueness of minimizers for $E_\ell$ (resp. $E_{n\ell}$) and another dealing directly with the equations using comparison arguments with the maximum principle
(using sub and supersolutions to prove existence of a solution). 
This gives two different approaches for the alternating method introduced below 
(minimizing vs. solving equations). In our numerical schemes we will solve 
the resulting equations. 
\end{remark}

With these two operators we define, for a given
$u_0$, the alternating method $(M)$ given by the sequence
$$(M) \, \begin{cases}
v_{n+1} &= L_{n \ell}^f (u_{n}) \\
u_{n+1} &= L_{\ell}^f (v_{n+1}).
\end{cases}
$$
That is, we start with $u_0$, a given function defined in the local region, then we solve our problem in the nonlocal region to obtain
$v_{1} = L_{n \ell}^f (u_{0})$ and then we solve in the local region to obtain 
$u_{1} = L_{\ell}^f (v_{1})$ and then we iterate this procedure. 

\begin{remark}\label{rem:f=0}
Convergence of $(M)$ to the solution $(u,v)$ of our problem with a general $f$, is equivalent to prove convergence to zero for the case $f=0$. Indeed, renaming $u_n=u_n-u$ and $v_n=v_n-v$ we see that they must verify
$$
v_{n+1} = L_{n \ell} (u_{n})
$$
$$
u_{n+1} = L_{\ell} (v_{n+1}).
$$
\end{remark}
Following \cite{Lions}, we introduce an appropriate Hilbert space. In our case we take
$$
H = \left\{(u,v) \, : \, u \in H^1_{\partial \Omega \cap \partial \Omega_{\ell}}   
(\Omega_{\ell})  \,  \mbox{ and } \,
v \in L^2 (\Omega_{n\ell})   \right\}\,
$$
with the norm 
$$
\begin{array}{l}
\displaystyle
\| (u,v) \|_{H}^2 =
\int_{\Omega_\ell} \frac{|\nabla u(x)|^2}{2} {\rm d}x + \frac{1}{2}
\int_{\Omega_{n\ell} }  
\int_{\Omega_{n\ell} } J(x-y)(v(y)-v(x))^2\, {\rm d}y{\rm d}x \\[10pt]
\qquad \qquad \quad \displaystyle 
 + \frac{1}{2}\int_{\Omega_{n\ell}}\int_{\Omega_{\ell}} J(x-y)(u(y)-v(x))^2\, {\rm d}y{\rm d}x +
 \frac{1}{2}
\int_{_{\Omega_{n\ell}} }  v(x)^2
\int_{\Omega^c} J(x-y)\, {\rm d}y{\rm d}x.
 \end{array}
$$
\begin{lemma}\label{lemma:equivH}
The space $H$ is a Hilbert space, with a norm equivalent to that of the product space  $H^1 (\Omega_\ell) \times L^2 (\Omega_{n\ell})$.
\end{lemma}
\begin{proof}
The obvious inner product associated to $\| (u,v) \|_{H}$ is
\begin{equation}\label{eq:pi}
 \begin{array}{l}
\displaystyle
\langle  (u,v), (z,w) \rangle_{H} =  
\frac12 \int_{\Omega_\ell}  \nabla u(x) \nabla z(x) {\rm d}x + \frac{1}{2}
\int_{_{\Omega_{n\ell}} }  v(x)w(x)
\int_{\Omega^c} J(x-y)\, {\rm d}y{\rm d}x\\[10pt]
\displaystyle 
\qquad \qquad \qquad \qquad \quad + \frac{1}{2}
\int_{\Omega_{n\ell} }  
\int_{\Omega_{n\ell}} J(x-y)(v(y)-v(x))(w(y)-w(x)) \, {\rm d}y{\rm d}x \\[10pt]
\qquad \qquad \qquad \qquad \quad \displaystyle 
 + \frac{1}{2}\int_{\Omega_{n\ell}}\int_{\Omega_{\ell}} J(x-y)(u(y)-v(x))(z(y)-w(x)) \, {\rm d}y{\rm d}x.
 \end{array}
\end{equation}

On the other hand, using the present notation, Lemma 3.3 in \cite{ABR} says that
$$
\int_{\Omega_\ell} \frac{|\nabla \tilde u(x)|^2}{2} {\rm d}x + \frac{1}{2}
\int_{\Omega_{n\ell} }  
\int_{\mathbb{R}^n } J(x-y)(\tilde u(y)-\tilde u(x))^2\, {\rm d}y{\rm d}x \ge
C\|\tilde u\|_{L^2(\Omega)},
$$
from which we obtain
$$\| (u,v) \|_{H}\ge C\left(\|u\|_{H^1(\Omega_\ell)}+\|v\|_{L^2(\Omega_{n\ell})}\right).$$ 
The other inequality, and therefore the equivalence, follows easily using the properties of $J.$ \end{proof}

Thanks to Remark \ref{rem:f=0} to show the convergence of the iterations we can work with $f=0$. Let us define the following two subspaces of $H$,
\begin{align*}
V_1 &= \left\{  (u,v) : u=L_{\ell}(v) \mbox{ with }  v\in L^2(\Omega_{n\ell}) \right\}, \\
V_2 & = \left\{  (u,v) : v=L_{n\ell}(u) \mbox{ with }  u\in H^1_{\partial \Omega_{\ell}\cap \partial \Omega}(\Omega_{\ell}) \right\}.
\end{align*}

\begin{lemma} \label{lema-suma-directa}
It holds that
$$
H = V_1 \oplus V_2.
$$
\end{lemma}

\begin{proof}
First notice that for $(u,v)\in V_1\cap V_2$ it should be
$u=L_{\ell}(v)$ and $v=L_{n\ell}(u)$. It means that $(u,v)$ solves the coupled system
\eqref{eq:main.Dirichlet.local.2277} and \eqref{eq:main.Dirichlet.nonlocal.2277} with $f=0$ and therefore $(u,v)=(0,0)$. This shows that $$V_1\cap V_2 = \{(0,0)\}.$$

Let now $(\hat{u},\hat{v})\in H$,  we claim that there exists a pair of functions $v\in L^2(\Omega_{\ell})$ and $u\in H^1_{\pp}(\Omega_\ell)$ such that
\begin{equation}\label{eq:suma}
(\hat{u},\hat{v})=(L_\ell(v),v)+(u,L_{n\ell}(u)),
\end{equation}
which would show that $$H=V_1+V_2.$$ Writting \eqref{eq:suma} as
\begin{align*}
 \hat u=& L_\ell(v)+u \\
 \hat v=& v + L_{n\ell}(u)
\end{align*}
we see that it is enough to show existence of solutions in $L^2(\Omega_{n\ell})$ for
\begin{equation}\label{eq:fredholm}
v-L_{n\ell}(L_{\ell} (v))=\hat v -L_{n\ell} (\hat u),
\end{equation}
since calling $u=\hat u-L_{\ell} (v)\in H^1_{\pp}(\Omega_{\ell})$, the pair $(u,v)$ solves \eqref{eq:suma}. Existence (and uniqueness) of a solution $v$ for \eqref{eq:fredholm} follows from the Fredholm alternative for the operator $I-K$ with $K=L_{n\ell}\circ L_{\ell}$. Indeed, since  $L_{\ell}$ is compact and $L_{n\ell}$ continuous we see $K:L^2(\Omega_{n\ell})\to L^2(\Omega_{n\ell})$ is compact and if $v$ is a solution of $v-L_{n\ell}(L_{\ell} (v))=0$ then $(L_{\ell} (v),v)$ is the unique solution of
the coupled system
\eqref{eq:main.Dirichlet.local.2277} and \eqref{eq:main.Dirichlet.nonlocal.2277} with $f=0$ and therefore $(L_{\ell} (v),v)=(0,0)$ which says in particular $v=0$. Fredholm alternative gives existence and uniqueness for $[I-K](v)=\hat w\in L^2(\Omega_{n\ell})$ and hence for \eqref{eq:fredholm}. The lemma follows.
\end{proof}
Now, let us define the linear operators
$$
P_i : H \mapsto H, \qquad 1\le i\le 2,
$$
given by
$$
P_1 (u,v) = (L_\ell (v), v),
$$
and
$$
P_2 (u,v) = (u, L_{n\ell} (u)).
$$

\begin{lemma} \label{proyection-i}
$P_i$ is the orthogonal projection
onto  $V_i$ w.r.t. the inner product \eqref{eq:pi}.
\end{lemma}

\begin{proof}
We only show the case $P_1$ since the other one is similar.

All we need to prove is the orthogonality relation $(u,v)-P_1(u,v)=(u-L_\ell(v),0)\perp V_1$ that follows easily from the expression  \eqref{eq:pi}, indeed
$$
\begin{array}{l}
\displaystyle
\langle  (u-L_\ell(v),0), (L_\ell(w),w) \rangle_{H}=
\frac12 \int_{\Omega_\ell}  \nabla [u-L_\ell(v)](x) \nabla L_\ell(w)(x) {\rm d}x +  \\[10pt]
\qquad \qquad \displaystyle
 + \frac{1}{2}\int_{\Omega_{n\ell}}\int_{\Omega_{\ell}} J(x-y)[u-L_\ell(v)](y)(L_\ell(w)(y)-w(x)) \, {\rm d}y{\rm d}x=0,
 \end{array}
 $$
 where the last identity follows by using the fact that $L_\ell(w)$ is a weak solution of
 \begin{equation}  \label{eq:ecuL}
\begin{cases}
\displaystyle 0=\Delta L_\ell(w) (x) + \int_{\Omega_{n\ell}} J(x-y)(w(y)-L_\ell(w)(x))\,{\rm d} y, &x\in \Omega_\ell,
\\[7pt]
\displaystyle \partial_\eta L_\ell(w)(x)=0,\qquad & x\in  \partial \Omega_\ell \cap \Omega,  \\[7pt]
L_\ell(w)(x)= 0, & x \in \partial \Omega \cap \partial \Omega_\ell,
\end{cases}
\end{equation}
and  taking $\phi=u-L_\ell(v)\in H^1_{\pp}(\Omega_{\ell})$ as a test function.
\end{proof}
Noticing that $V_1$ and $V_2$ are closed,   an immediate corollary of Theorem I.1 in \cite{Lions} says that our alternating method converges strongly in the norm of our Hilbert space.
\begin{theorem} \label{th:conv_cont}
Under the assumptions $(J1)$, $(J2)$, $(1)$, $(2)$ and $(P1)$ and given $f \in L^2 (\Omega)$, the alternating method $(M)$ converges to the  unique minimizer $(u,v)$ of $E_i$ (equiv. to the unique weak solution of the system \eqref{eq:main.Dirichlet.local.2277}, \eqref{eq:main.Dirichlet.nonlocal.2277}). That is,
$$
u_n \to u \qquad \mbox{in } H^1 (\Omega_\ell),
$$
$$
v_n \to v \qquad \mbox{ in } L^2 (\Omega_{n\ell}).
$$

Moreover, since $V_1+V_2=H$, the method is geometricaly convergent. That is, there exists $0<\delta<1$ such that
$$
\|u_n - u\|_{H^1_{\pp}(\Omega_\ell)},\|v_n - v\|_{L^2(\Omega_{n\ell})}\le C \delta^n.
$$
\end{theorem}
\begin{remark}
In stark contrast with classical domain decomposition applied to local PDEs under Dirichlet boundary conditions, no \emph{overlapping}  of the underlying subdomains is needed in order to get geometric convergence.
\end{remark}

\begin{remark} \label{remark.pepe} {\rm
The previous approach shows the benefits of understanding the problem considered in
\cite{ABR}  as a system composed by a local equation
for the first component, $u$, in $\Omega_{\ell}$
whose main part is the Laplacian,
\begin{equation}  \label{eq:main.Dirichlet.local.227788}
\begin{cases}
\displaystyle - f(x)=\Delta u (x) - \left( \int_{\Omega_{n\ell}} J(x-y) {\rm d} y \right)  u(x)+ \int_{\Omega_{n\ell}} J(x-y)v(y)\,{\rm d} y, &x\in \Omega_\ell,
\\[7pt]
\displaystyle \partial_\eta u(x)=0,\qquad & x\in  \partial \Omega_\ell \cap \Omega,  \\[7pt]
u(x)= 0, & x \in \partial \Omega \cap \partial \Omega_\ell,
\end{cases}
\end{equation}
and a nonlocal equation for the second component $v$ in $\Omega_{n\ell}$,
\begin{equation}  \label{eq:main.Dirichlet.nonlocal.227788}
\left\{
\begin{array}{l}
\displaystyle \!\! - f(x) =  2 \! \int_{\Omega_{n\ell}} J(x-y)  (\! v(y)-v(x) \!){\rm d}y - \!\! \left(\int_{\mathbb{R}^N\setminus \Omega_{n\ell}} J(x-y) {\rm d}y \right) v(x)
\\[7pt] \displaystyle \qquad \qquad \qquad \qquad \qquad \qquad \qquad\qquad\qquad\qquad +
 \!\! \int_{ \Omega_{\ell}} J(x-y)u(y) {\rm d}y, \qquad
\, x \in \Omega_{n\ell}, \\[7pt]
\displaystyle \!\! v(x)= 0,\qquad  x\in \mathbb{R}^N\setminus \Omega.
\end{array} \right.
\end{equation}
}
\end{remark}

Next, we show that the iterations provide a monotone approximation to the
solution when the initial function $u_0$ is the local part of a subsolution.

To this end we need the following definition.

 \begin{definition}\label{def:sub-super}
 We call $(\underline{u},\underline{v})$ a subsolution to
 \eqref{eq:main.Dirichlet.local.2277}--\eqref{eq:main.Dirichlet.nonlocal.2277}
 if it verifies
 \begin{equation}  \label{eq:main.Dirichlet.local.2277sub}
\begin{cases}
\displaystyle - f(x) \leq \Delta \uu (x) + \int_{\Omega_{n\ell}} J(x-y)(\vv(y)-\uu(x))\,{\rm d} y, &x\in \Omega_\ell,
\\[7pt]
\displaystyle \partial_\eta \uu(x)\leq 0,\qquad & x\in  \partial \Omega_\ell \cap \Omega,  \\[7pt]
\uu(x)\leq  0, & x \in \partial \Omega \cap \partial \Omega_\ell,
\end{cases}
\end{equation}
 in $\Omega_{\ell}$ and
\begin{equation}  \label{eq:main.Dirichlet.nonlocal.2277sub}
\left\{
\begin{array}{l}
\displaystyle \!\! - f(x) \leq \!\! \int_{\mathbb{R}^N\setminus \Omega_{n\ell}} J(x-y)(\! \uu(y)-\vv(x)\! ) {\rm d}y
+ 2 \! \int_{\Omega_{n\ell}} J(x-y)  (\! \vv(y)-\vv(x) \!){\rm d}y, \\[7pt]
\qquad\qquad\qquad\qquad\qquad\qquad\qquad\qquad\qquad\qquad \qquad\qquad\qquad
\, x \in \Omega_{n\ell}, \\[7pt]
\displaystyle \!\! \uu(x) \leq 0,\qquad  x\in \mathbb{R}^N\setminus \Omega.
\end{array} \right.
\end{equation}

 Similarly, a supersolution $(\overline{u},\overline{v})$ is a function that verifies
 \eqref{eq:main.Dirichlet.local.2277sub}--\eqref{eq:main.Dirichlet.nonlocal.2277sub}
 with the reverse inequalities,
 \end{definition}

For  \eqref{eq:main.Dirichlet.local.2277} and for \eqref{eq:main.Dirichlet.nonlocal.2277}
we have a comparison principle. For the proofs we refer to \cite{ElLibro} and \cite{evans}.
For the local part of our problem we have the following lemma.

\begin{lemma} \label{lema-compar-local} Fix $v$. Let $\underline{u}$ be a subsolution to
\eqref{eq:main.Dirichlet.local.2277}, that is, function that verifies
\eqref{eq:main.Dirichlet.local.2277sub} and $u$ a solution to
\eqref{eq:main.Dirichlet.local.2277} with the same $v$. Then
\begin{equation} \label{compar.local}
\underline{u} (x) \leq u (x), \qquad x \in \Omega_{\ell}.
\end{equation}
\end{lemma}

The corresponding statement for the nonlocal part reads as follows.

\begin{lemma}  \label{lema-compar-nonlocal} Fix $u$. Let $\underline{v}$ be a subsolution to
\eqref{eq:main.Dirichlet.nonlocal.2277}, that is, function that verifies
\eqref{eq:main.Dirichlet.nonlocal.2277sub} and $v$ a solution to
\eqref{eq:main.Dirichlet.nonlocal.2277} with the same $u$. Then
\begin{equation} \label{compar.nonlocal}
\underline{v} (x) \leq v (x), \qquad x \in \Omega_{n\ell}.
\end{equation}
\end{lemma}

With these two lemmas we can show the desired monotonicity for the sequence generated
by the iterations.

\begin{theorem} \label{teo.monotono}
Let $(\underline{u},\underline{v})$ be a subsolution to
\eqref{eq:main.Dirichlet.local.2277}--\eqref{eq:main.Dirichlet.nonlocal.2277}. Then, the sequences $u_n$, $v_n$ generated by
$$(M) \, \begin{cases}
v_{n+1} &= L_{n \ell}^f (u_{n}) \\
u_{n+1} &= L_{\ell}^f (v_{n+1}).
\end{cases}
$$
with
$$
u_0 = \underline{u}
$$
is monotone, that is, it holds that
$$
\underline{u}(x) \leq u_n (x) \leq u_{n+1} (x) \leq u(x), \qquad x \in \Omega_{\ell},
$$
and
$$
\underline{v}(x) \leq v_n (x) \leq v_{n+1} (x)\leq v(x), \qquad x \in \Omega_{n\ell}.
$$
\end{theorem}

\begin{proof} Let us start with $v_1$. From the comparison lemma
\ref{lema-compar-nonlocal} with $\underline{u}$ fixed we get that
$$
v_1(x) \geq \underline{v} (x), \qquad x \in \Omega_{n\ell},
$$
recall that $v_1$ is a solution to \eqref{eq:main.Dirichlet.nonlocal.2277}
with $\underline{u}$.

Now, for $u_2$ (that is a solution to \eqref{eq:main.Dirichlet.local.2277}
with $v_1$, using the comparison lemma for the local part of the problem,
Lemma \ref{lema-compar-local}, we get
$$
u_1(x) \geq \underline{u} (x)=u_0 (x), \qquad x \in \Omega_{\ell}.
$$

From our previous arguments, the proof of the general monotonicity property for the sequences follows by
induction.
\end{proof}

A particular case where this result can be applied is when the source term $f$
is nonnegative,
$$
f(x) \geq 0, \qquad x \in \Omega.
$$
Then we get that $\underline{u} = \underline{v} = 0$ is a subsolution to
\eqref{eq:main.Dirichlet.local.2277}--\eqref{eq:main.Dirichlet.nonlocal.2277}
and therefore the sequences $u_n$, $v_n$ are monotone and nonnegative.

In the general case, assuming that $f$ is bounded, we have that
$\underline{u} = \underline{v} = -C$ is a subsolution to
\eqref{eq:main.Dirichlet.local.2277}--\eqref{eq:main.Dirichlet.nonlocal.2277}
for a large constant $C$.
Therefore, we can ensure that the sequences $u_n$, $v_n$ are monotone
for every bounded $f$ just by choosing the initial function $u_0$ as a large
negative constant.

\begin{remark} From the results in \cite{ABR} we know that when $f$ is continuous and the kernel $J$ is smooth
the solution is continuous up to the boundary in both the local and the nonlocal domains (but not
necessarily across the interface). In this case if we start with a continuous
subsolution we obtain increasing sequences $u_n$, $v_n$ of continuous 
functions that converge to a continuous limit and hence we conclude that the
convergence of the iterations to the solution is uniform in this case.
\end{remark}

\begin{remark}
Similar computations prove that when we start with a supersolution instead of a subsolution the method
provides decreasing sequences. 
\end{remark}

\section{Abstract Numerical Setting}
\label{sect-numerical}

The approach developed in the continuous case can be straightforwardly applied at the discrete level. For the sake of completeness and before giving some numerical examples, we provide the standard basic theoretical details.

For any pair of conforming Galerkin approximation spaces
$$
\U\subset H^1_{\pp}(\Omega_\ell) \qquad \mbox{and}  \qquad
\V\subset L^2(\Omega_{n\ell})
$$
with inherited norms,  we consider the discrete operator $L_{\ell,D}^f:\V\to \U$,
$$
L_{\ell,D}^f (v_D) = u_D,
$$
where, for a given $v_D\in \V$,  $u_D$ stands for the solution of
$$
\int_{\Omega_{\ell}}\nabla u_D (x) \nabla \bar u_D  + \int_{\Omega_{\ell}}\int_{\Omega_{n\ell}} J(x-y)(v_D(y)-u_D(x))\bar u_D(x)\,{\rm d} y\,{\rm d} x=\int_{\Omega_{\ell}}f(x)\bar u_D(x) ,
$$
for all $\bar u_D\in \U$. In a similar way, we define $L_{n\ell,D}^f:\U\to \V$,
$$
L_{n\ell,D}^f (u_D) = v_D,
$$
where, for a given $u_D\in \U$,  $v_D$ stands for the solution of
\begin{align*}  \! &\int_{\Omega_{n\ell}} \int_{\Omega_{n\ell}} J(x-y)  (\! v_D(y)-v_D(x) \!)(\! \bar v_D(y)-\bar v_D(x) \!){\rm d}y {\rm d}y + \\ &\int_{\Omega_{n\ell}}\bar v_D(x)\int_{\Omega_{n\ell}^c} J(x-y)(\! u_D(y)-v_D(x)\! ) {\rm d}y {\rm d}x
=\int_{\Omega_{n\ell}} f(x)\bar v_D(x),
\end{align*}
for all $\bar v_D\in \V$.  Both operators are well defined, as one can see from the proofs of Lemma \ref{lema.est.u} and  Lemma \ref{lema.est.v} respectively, and the whole theory of the previous section applies for the following discrete alternating method: for  $u_{D,0}\in \U$, define  the sequence
$$(M_D) \, \begin{cases}
v_{D,n+1} &= L_{n \ell}^f (u_{D,n}) \\
u_{D,n+1} &= L_{\ell}^f (v_{D,n+1}).
\end{cases}
$$
In particular, as before, we can prove geometric convergence in the strong norm to the weak solution of the coupled system posed in the variational  product space $\U\times \V$, that is, to the Galerkin approximation of the original coupled system. Taking, in particular,  finite element spaces $\U_h$ and $\V_{\tilde{h}}$ associated to not necessarily equal mesh parameters ($h$ and ${\tilde{h}}$) and  with possibly different polynomial degree based elements, we can easily build discrete approximations. Moreover, calling $u_{h}$ and $v_{\tilde{h}}$ the solutions in $\U_h$ and $\V_{\tilde{h}}$, convergence of
$(u_h,v_{\tilde{h}})\to (u,v)$ for $h,\tilde{h}\to 0$ follows by standard arguments. Actually, since $(u,v)$ is a weak solution of the coupled system \eqref{eq:main.Dirichlet.local.2277},
\eqref{eq:main.Dirichlet.nonlocal.2277}, proceeding in the usual way, we get the orthogonality of the error $e_{h,{\tilde{h}}}=(u-u_h,v-v_{\tilde{h}})$ in the scalar product $\langle\cdot,\cdot\rangle_H$ and as consequence, the best approximation property in the natural norm $\|\cdot\|_H$, 
$$
\begin{array}{l}
\displaystyle
\|(u-u_h,v-v_{\tilde{h}})\|_H^2=\langle (u-u_h,v-v_{\tilde{h}}),(u-\pi_{\U_h}(u),v-\pi_{\V_{\tilde{h}}}(v))\rangle_H \\[10pt]
\displaystyle \qquad \qquad \le \|(u-u_h,v-v_{\tilde{h}})\|_H\|(u-\pi_{\U_h}(u),v-\pi_{\V_{\tilde{h}}}(v))\|_H
\end{array}
$$
for appropriate interpolation or projection operators $\pi_{\U_h}$, $\pi_{\V_{\tilde{h}}}$ onto $\U_{h}$ and $\V_{\tilde{h}}$ respectively.
Using the equivalence of norms stated in Lemma \ref{lemma:equivH}, we  get
$$
\begin{array}{l}
\displaystyle
\|(u-u_h,v-v_{\tilde{h}})\|_H
\le \|(u-\pi_{\U_h}(u),v-\pi_{\V_{\tilde{h}}}(v))\|_H \\[10pt]
\qquad \displaystyle \le C\left(\|u-\pi_{\U_h}(u)\|_{H^1(\Omega_\ell)}+\|v-\pi_{\V_{\tilde{h}}}(v)\|_{L^2(\Omega_{n\ell})}\right).
\end{array}
$$
Naturally, at this point, we need to assume some regularity for the  solutions, if we want to guarantee convergence with order. Since our nonlocal kernels are quite general, a priori regularity results are not easy to provide without assuming further ad-hoc hypotheses. Consider for instance  a smooth and bounded source function $f$, then we see from equations \eqref{eq:main.Dirichlet.local.227788} and \eqref{eq:main.Dirichlet.nonlocal.227788} that smooth compactly supported kernels $J$ provide smooth (in their corresponding domains) solutions $u$ and $v$ (not necessarily continuous across the interface). Indeed, once the existence of solutions is proved in $H$ we readily  notice from \eqref{eq:main.Dirichlet.local.227788} that $u$ is actually better than $H^1 (\Omega_\ell)$, let say $H^2(\Omega_\ell)$, for smooth $\Omega_\ell$ (or  even for a convex polygon in $\RR^2$). This in turn says, from
\eqref{eq:main.Dirichlet.nonlocal.227788}, that $v$ is at least as smooth as $f$ and $J$.

Finite element approximations for  nonlocal equations require several computational considerations. For kernels that are not compactly supported, unbounded regions of integration should be taken into account. This difficulty can be addressed, at least for radial kernels, by taking appropriate exterior meshes (see for instance \cite{ABB} for the case of the fractional Laplacian). Moreover, for highly singular kernels, specific quadrature techniques are needed, see again \cite{ABB}. In this paper, however, only  integrable kernels are considered and in order to keep implementation issues simple we restrict our attention to the case of a bounded and compactly supported $J$.  We give next a shallow description of the FE setting where, in addition,  we assume a single mesh parameter $h$ for the sake of simplicity.   

Calling  $B(0,R)$, the ball centered at the origin and radius $R$, such that  $supp(J)\subset B(0,R)$, we need to  mesh an enlarged set, $\Omega_{R}=\{x\in \RR^n: d(x,\Omega)<R\}$, in order to compute nonlocal interactions within the horizon of $J$.  Furthermore,   we consider a finite element triangulation
$\cup_{T\in \mathcal{T}_h}T=\bar\Omega_{R}$ which is (simultaneously) admissible  for $\Omega_{\ell}$, $ \Omega_{n\ell}$ and $\Omega_{R}\setminus \Omega$. $\mathcal{T}_h$ is made up of elements $T$ of diameter $h_T$ with  inner radius $\rho_T$ and that are regular in the  standard way: $\exists \sigma > 0$ such that  $h_T \leq \sigma \rho_T$ for every $ T \in \mathcal{T}_h$. At this point, the degree of the underlying polynomial spaces can be chosen in different ways, in particular they should not necessarily agree in both, the local and nonlocal domains. The assembling involves local terms associated to the  laplacian and nonlocal ones of the kind
\begin{equation}
\label{eq:Inmij}
I_{m,n}^{i,j}:=\int_{T_m}\int_{T_n}J(x-y)(\phi_i(x)-\phi_i(y))(\phi_j(x)-\phi_j(y))\,dx\,dy,
\end{equation}
for pairs of elements $T_m,T_n\in \mathcal{T}_h$ and finite element basis functions $\phi_i,\phi_j$. The needed basis functions are those corresponding to nodes belonging to $\Omega$, since we are dealing with homogenous boundary (or complementary) conditions. However, notice that there are nontrivial contributions of $I_{m,n}^{i,j}$  involving elements contained in $\bar \Omega_{R}\setminus \Omega$. Indeed, consider for instance the following term of \eqref{eq:Inmij},
$$
\int_{T_m}\int_{T_n}J(x-y)\phi_i(x)\phi_j(x)\,dx\,dy,
$$
for pairs of elements such that $T_m\subset \bar \Omega_{R}\setminus \Omega$, $T_n\subset \Omega_{n\ell}$. If the set $K= supp(\phi_i)\cap supp(\phi_j) \cap T_n$ is nonempty and if the distance between $K$ and $T_m$ is less than $R$, then
$$
\int_{T_m}\int_{T_n}J(x-y)\phi_i(x)\phi_j(x)\,dx\,dy \neq 0.
$$
Also notice that since  the support of $J(x-y)\phi_i(x)\phi_j(y)$ can be a proper subset of $T_m\times T_n$, some care is needed in handling the numerical integration. 
 
\section{A 1D Numerical Example}

The following 1D numerical example takes continuous $P_1$ for all the involved elements although piecewise constants could be fine in the nonlocal subdomain $\Omega_{n\ell}$, as a conforming approximation of $L^2(\Omega_{n\ell})$. We consider the domains: $$\Omega_{n\ell}=(-1,0), \qquad \Omega_{\ell}=(0,1),$$
and the source $$f=(1-x)^4.$$ For simplicity, the kernel $J$ is piecewise linear 
 $$ J(x) = \frac{ (1-\frac{|x|}{\delta}) }{\delta} \chi_{[- \delta, \delta]}(x)$$
with $\delta = 0.5$ or $\delta=0.3$. 

In Figure \ref{fig:example1d} we show, for a fixed $h$, with $\delta=0.5$, the FE limit solution as well as some intermediate iterates.  As it was proved the iterates converge geometrically.

\begin{figure}
\begin{center}
    \includegraphics[scale=.23]{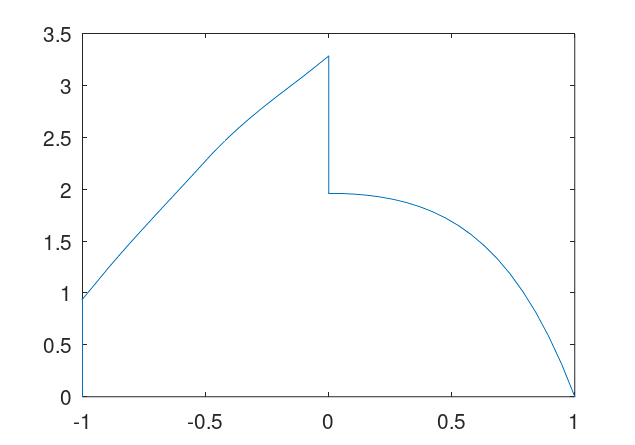}
  \includegraphics[scale=.23]{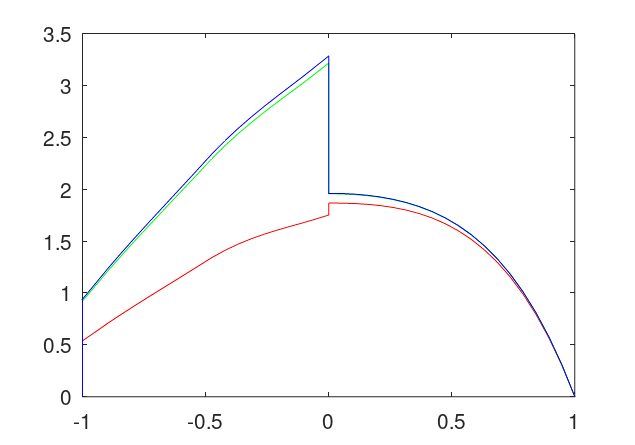}
  \includegraphics[scale=.27]{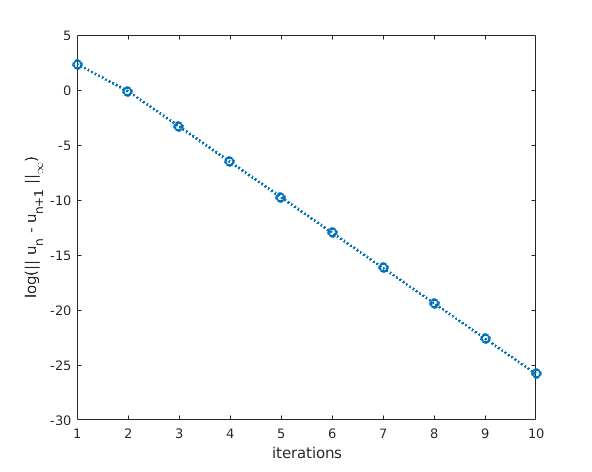}
  
    		\caption{{Limit Solution (left) and  three Schwarz $(M)$ iterates (middle) for the example. We use $n=100$ nodes and continuous piecewise linear FE. The discontinuity across the domains is visible. Only a few iterations are needed as we count with geometric convergence (right). Notice also that the convergence is monotone increasing.}}
\end{center}
\label{fig:example1d}
\end{figure}

\begin{remark}[Monotone Approximations]
The described monotonic behavior for the continuous level can be expected in the discrete counterpart (for a fixed discretization) if the maximum principle holds for the discrete version of both the local and nonlocal operators. In particular, for 1D problems and uniform meshes, this is the case in the local part if we use the $P_1$ FE space (with mass lumping for the lowest order term). In our numerical example this behavior is apparent (see Figure \ref{fig:example1d}).
\end{remark} 

\begin{remark}[Parallel Schwarz]
The parallel Schwarz version for the present case of \emph{two subdomains} (i.e. one local and one nonlocal) reads:
for a given
$u_0$ and $v_0$, we define the sequence 
$$(M_P) \, \begin{cases}
v_{n+1} &= L_{n \ell}^f (u_{n}) \\
u_{n+1} &= L_{\ell}^f (v_{n}).
\end{cases}
$$
 $(M_P)$ is obviously relatable to the classical Schwarz $(M)$ since in this simple two-subdomain case the sequence computed with $(M_P)$ on a subdomain coincides every two steps  with the sequence computed on that subdomain by $(M)$. From the matrix counterpart they can be seen, in the discrete level, as block Jacobi vs. block Gauss-Seidel iterative methods. For a \emph{multi-domain case}, on the contrary,   $(M_P)$ has advantages versus the standard $(M)$ method and although we did not treat the theory we give a numerical example in Figure \ref{fig:parvssec}.   We focus only on a purely nonlocal case with three subdomains and $f(x)=1$. Remarkably, the nonlocal nature of the problem brings some benefits even for the standard alternating method $(M)$ since for that case and again not overlapping is needed.

\begin{figure}\begin{center}
    \includegraphics[scale=.35]{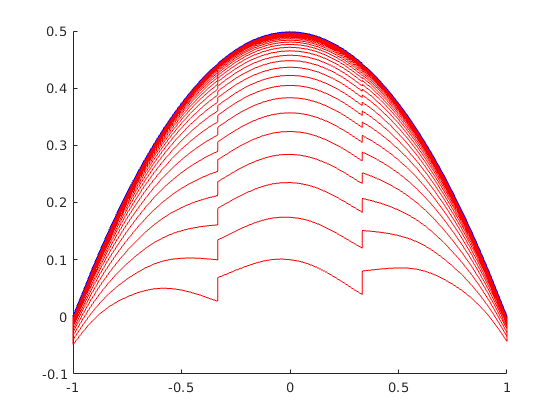}
  \includegraphics[scale=.35]{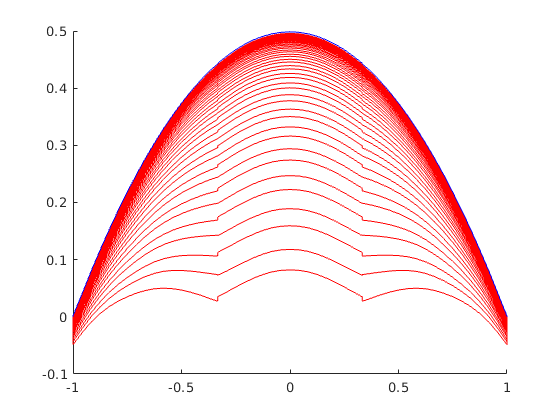}
    		\caption{{Alternating (left) vs. Parallel (right) Schwarz. Only the nonlocal model, with three subdomains. As before, no superposition is needed, since the variational space is $L^2(\Omega_{n\ell})$.  In blue, the limit solution with $f(x)=1$, $\delta=0.3$ and $n=100$ nodes is displayed.}}
\end{center}
\label{fig:parvssec}
\end{figure}

\end{remark}


\bibliographystyle{plain}

\begin{thebibliography}{99}

\bibitem{ABB} Acosta G.; Bersetche, F. M.; Borthagaray, J. P. \emph{A short FE implementation for a 2d homogeneous Dirichlet problem of a fractional Laplacian.} Comput. Math. Appl. 74, No. 4, (2017) 784--816.

\bibitem{ABR} Acosta, G.; Bersetche, F. M.; Rossi, J. D. \emph{Local and nonlocal energy-based coupling models}. To appear in SIAM Jour. Math. Anal.

\bibitem{Peri1} Azdoud, Y.; Han, F.; Lubineau, G. {\it A morphing framework to couple non-local and local anisotropic
continua}. Inter. J. Solids Structures 50(9),  (2013), 1332--1341.

\bibitem{ElLibro}
   Andreu-Vaillo, F.;  Toledo-Melero, J.; Mazon, J. M.;  
  Rossi, J. D.
\newblock \emph{Nonlocal diffusion problems}.
\newblock Number 165. American Mathematical Soc., 2010.


\bibitem{Peri2} Badia, S.; Bochev, P.; Lehoucq, R.; Parks, M.; Fish, J.; Nuggehally, M.A.; Gunzburger, M. {\it A forcebased
blending model for atomistic-to-continuum coupling}. Inter. J. Multiscale
Comput. Engineering, 5(5), (2007), 387--406.

\bibitem{Peri3} Badia, S.; Parks, M.; Bochev, P.; Gunzburger, M.; Lehoucq, R. {\it On atomistic-to-continuum coupling
by blending}. Multiscale Modeling Simulation, 7(1), (2008), 381--406.


\bibitem{BCh} Bates, P.; Chmaj, A. \emph{An integrodifferential
model for phase transitions: stationary solutions in higher
dimensions}. J. Statist. Phys. 95 (1999), no.\,5--6, 1119--1139.

\bibitem{Bere} Berestycki, H., Coulon, A.-Ch.; Roquejoffre, J-M.; Rossi, L. 
{ \it The effect of a line with nonlocal diffusion on Fisher-KPP propagation.} Math. Models Meth.  Appl. 
Sciences, 25.13, (2015), 2519--2562.


\bibitem{brezis}
Brezis, H. \emph{Functional analysis, Sobolev spaces and partial differential equations}. Springer Science \& Business Media, 2010.




\bibitem{CF} Carrillo, C.; Fife, P. \emph{Spatial effects in discrete generation population models}. J. Math. Biol. 50 (2005), no.\,2, 161--188.


\bibitem{ChChRo} Chasseigne, E.; Chaves, M.; Rossi, J.\,D. \emph{Asymptotic behavior for nonlocal diffusion equations}. J. Math. Pures Appl. (9) 86 (2006), no.\,3, 271--291.


\bibitem{CERW} Cort\'azar, C.; M. Elgueta, M.; Rossi, J.\,D.; Wolanski, N. \emph{Boundary fluxes for non-local diffusion}.   J. Differential Equations 234 (2007), no.\,2, 360--390.

\bibitem{Cortazar-Elgueta-Rossi-Wolanski}  Cort\'azar, C.; M. Elgueta, M.; Rossi, J.\,D.; Wolanski, N. \emph{How to approximate the heat equation with Neumann boundary conditions by nonlocal diffusion problems}. Arch. Ration. Mech. Anal. 187 (2008), no.\,1, 137--156.


\bibitem{delia2} D'Elia, M.; Perego, M.; Bochev, P.; Littlewood, D. \emph{A coupling strategy for nonlocal and local diffusion models with mixed volume constraints and boundary conditions}. Comput. Math. Appl. 71 (2016), no.\,11, 2218--2230.

\bibitem{delia3} D'Elia, M.; Ridzal, D.; Peterson, K.\,J.; Bochev, P.; Shashkov, M. \emph{Optimization-based mesh correction with volume and convexity constraints}. J. Comput. Phys. 313 (2016), 455--477.


\bibitem{delia}
D'Elia, M.; Bochev, P. \emph{Formulation, analysis and computation of an optimization-based local-to-nonlocal coupling method}. arXiv:1910.11214.

\bibitem{SUR} 
D'Elia, M.; Li, X.; Seleson, P.; Tian, X.; Yu, Y. {\it A review of Local-to-Nonlocal coupling methods in nonlocal diffusion and nonlocal mechanics.}
arXiv:1912.06668.



\bibitem{DiPaola} Di Paola, M.; Giuseppe F.; M. Zingales. { \it Physically-based approach to the mechanics of strong non-local linear elasticity theory.} Journal of Elasticity 97.2 (2009): 103--130.


\bibitem{Du} Du, Q.; Li, X.\,H.; Lu, J.; Tian, X. \emph{A quasi-nonlocal coupling method for nonlocal and local diffusion models}. SIAM J. Numer. Anal. 56 (2018), no.\,3, 1386--1404.

\bibitem{evans}
Evans, L.C. \emph{Partial Differential Equations}. Second edition. Graduate Studies in Mathematics, 19. American Mathematical Society, Providence, RI, 2010.


\bibitem{F} Fife, P. \emph{Some nonclassical trends in parabolic and parabolic-like evolutions}. In \lq\lq Trends in nonlinear analysis'', 153--191, Springer, Berlin, 2003.


\bibitem{Gal} Gal, C. G.; Warma, M. {\it Nonlocal transmission problems with fractional diffusion and boundary conditions on non-smooth interfaces.} Communications in Partial Differential Equations, 42(4) (2017), 579--625.

\bibitem{Gander} Gander, M. J. {\it Schwarz methods over the course of time.} ETNA, Electron. Trans. Numer. Anal. 31, 228-255 (2008).

\bibitem{GQR} G\'{a}rriz, A.; Quir\'os, F.; Rossi, J. D. {\it Coupling local and nonlocal evolution equations.} Calc. Var. PDE, 
59(4), article 117, (2020), 1--25.


\bibitem{Han} Han, F., Gilles L.; { \it Coupling of nonlocal and local continuum models by the Arlequin approach.} Inter. Journal Numerical Meth. 
Engineering 89.6 (2012): 671--685.


\bibitem{Hutson} Hutson, V.; Martinez, S.; Mischaikow, K.; Vickers, G. T. {\it The evolution of dispersal}. J. Math. Biology, 47(6), (2003), 483--517.


\bibitem{Kri} Kriventsov, D. {\it Regularity for a local-nonlocal transmission problem}. Arch. Ration.
Mech. Anal. 217 (2015), 1103--1195.

\bibitem{Lions} P.L. Lions {\it On the Schwarz alternating method. I}  1st International Symposium on Domain Decomposition Methods for Partial Differential Equations, SIAM, Philadelphia, (1988), pp. 1-42.

\bibitem{MenDu} 
Mengesha, T. and Du, Q., {\it The bond-based peridynamic system with Dirichlet-type volume constraint}, Proc. Roy. Soc. Edinburgh Sect. A, (2014), Nro. 1, p.p. 161-186.


\bibitem{Schwarz} Scwharz, H. A. , \"Uber einen Grenz\"ubergang durch alternierendes Verfahren, Vierteljahrsschrift der Naturforschenden Gesellschaft in Z\"urich, 15 (1870), pp. 272-286.


\bibitem{Sel2} Seleson, P.; Gunzburger, M. {\it Bridging methods for atomistic-to-continuum coupling and their implementation}.
Comm. Comput. Physics, 7(4), (2010), 831. 

\bibitem{Sel3} Seleson, P.; Gunzburger, M.; Parks, M.L. {\it Interface problems in nonlocal diffusion and sharp transitions
between local and nonlocal domains}. Comput. Methods Appl. Mech. Engineering,
266, (2013), 185-204.


\bibitem{Sil} Silling, S. A.; {\it Reformulation of elasticity theory for discontinuities and long-range forces}. Jour. Mech. Physics Solids, 48(1), 
 2000, 175---209.
 

\bibitem{Sil1} Silling, S. A.; Lehoucq, R. B.; { \it Peridynamic theory of solid mechanics}. In Advances in applied mechanics (Vol. 44, pp. 73-168). Elsevier, 2010.


\bibitem{Sil2} S. A. Silling, M. Epton, O. Weckner, J. Xu, 
and E. Askari. {\it Peridynamic
states and constitutive modeling.} Journal of Elasticity, 
88:151-184, 2007.

\bibitem{Strick} Strickland, C.; Gerhard D.; Patrick D. S.; { \it Modeling the presence probability of invasive plant species with nonlocal dispersal.} Jour.
Math. Biology 69.2 (2014), 267--294.


\bibitem{W} Wang, X. \emph{Metastability and stability of patterns in a convolution model for phase transitions}. J. Differential Equations 183 (2002), no.\,2, 434-461.

\bibitem{Z} Zhang, L. \emph{Existence, uniqueness and exponential stability of traveling wave solutions of some integral differential equations
arising from neuronal networks}. J. Differential Equations 197 (2004), no.\,1, 162-196.


\end{thebibliography}

\end{document}